\documentclass[11pt,reqno,a4paper]{amsart}
\usepackage[T1]{fontenc}
\usepackage{lmodern}
\usepackage{amsmath,amssymb,mathtools,mathrsfs}
\usepackage{microtype}
\usepackage[hidelinks]{hyperref}
\hypersetup{pdftitle={On Mahler U-numbers under rational maps},
  pdfauthor={Diego Marques}}

\newtheorem{theorem}{Theorem}[section]
\newtheorem{proposition}[theorem]{Proposition}
\newtheorem{lemma}[theorem]{Lemma}
\newtheorem{corollary}[theorem]{Corollary}
\theoremstyle{remark}
\newtheorem{remark}[theorem]{Remark}

\newcommand{\Q}{\mathbb Q}
\newcommand{\R}{\mathbb R}
\newcommand{\Z}{\mathbb Z}

\newcommand{\U}{\mathscr U}
\newcommand{\Lio}{\mathscr L}
\DeclareMathOperator{\dist}{dist}
\DeclareMathOperator{\type}{type}
\numberwithin{equation}{section}

\title[Mahler types under rational maps]
{On Mahler $U$-numbers under rational maps}
\author{Diego Marques}
\address{Departamento de Matem\'atica, Universidade de Bras\'ilia,
Bras\'ilia, 70910-900, Brazil}
\email{diego@mat.unb.br}
\date{}

\begin{document}

\begin{abstract}
In 1906, Maillet proved that every nonconstant rational function over
$\Q$ maps Liouville numbers to Liouville numbers. Let $R$ be a
nonconstant rational function with real algebraic coefficients, of degree
$d$, and let $e$ be the degree of its minimal field of definition over
$\Q$. We prove that, for every $m\geq 1$, the possible types of $R(\xi)$, as
$\xi$ ranges over $U$-numbers of type $m$, are exactly the integers $r$
with $\lceil m/(ed)\rceil\leq r\leq em$. Moreover, every such type
occurs for some $\xi$ of type $m$ in every nonempty open interval.
Consequently, for $m\geq2$, the only rational functions with real
algebraic coefficients preserving the class of $U$-numbers of type $m$
are the M\"obius transformations over $\Q$.
\end{abstract}

\maketitle

\section{Introduction}

A real number $\xi$ is a \emph{Liouville number} if, for every integer
$N\geq 1$, there are infinitely many coprime integers $p,q$, with $q\geq 2$,
such that
\[
  0<\left|\xi-\frac pq\right|<q^{-N}.
\]
We denote the set of Liouville numbers by $\Lio$.

In 1906, Maillet~\cite[Chap.~III, pp.~32--33]{Maillet} proved that every
nonconstant rational function over $\Q$ maps $\Lio$ into itself. A natural
extension of this preservation problem arises from Mahler's
classification~\cite{Mahler1932}. Introduced in 1932, this classification
partitions the transcendental numbers into three disjoint classes, called
$S$-, $T$-, and $U$-numbers, according to their approximation by integer
polynomials. The $U$-numbers generalize Liouville numbers by replacing
rational approximation with approximation by algebraic numbers of bounded
degree.

We use Koksma's formulation in terms of approximation by algebraic
numbers~\cite{Koksma}. The classifications of Mahler and Koksma are
equivalent and assign the same type to each $U$-number; see
also~\cite[Chap.~3]{Bugeaud}. For $\xi\in\R$ and an integer $n\geq 1$,
let $w_n^*(\xi)$ denote the supremum of the real numbers $w$ for which
\[
  0<|\xi-\alpha|\leq H(\alpha)^{-w-1}
\]
holds for infinitely many real algebraic numbers $\alpha$ of degree exactly $n$. Here $H(\alpha)$ is the maximum of the absolute values of the
coefficients of its primitive minimal polynomial in $\Z[X]$. We use the exact-degree convention; this gives the usual $U$-type.

A transcendental real number $\xi$ is called a \emph{$U_m$-number} if $m$
is the least positive integer such that $w_m^*(\xi)=\infty$. We denote the
class of real $U_m$-numbers by $\U_m$. Thus $\U_1=\Lio$, and, for
$m\geq 2$,
\[
  \xi\in\U_m
  \quad\Longleftrightarrow\quad
  w_m^*(\xi)=\infty
  \quad\text{and}\quad
  w_j^*(\xi)<\infty\quad(1\leq j<m).
\]
The integer $m$ is called the \emph{type} of $\xi$ and is denoted by
$\type(\xi)$. For a rational function $R=P/Q$ in reduced form, we write
$\deg R=\max\{\deg P,\deg Q\}$.

In this notation, Maillet's theorem states that
\[
  R(\U_1)\subseteq\U_1
  \qquad\text{for every nonconstant }R\in\Q(x).
\]

M\"obius transformations over $\Q$ preserve every $U$-class: in fact, they
preserve the Koksma exponents, and hence $M(\U_m)=\U_m$ for every
$m\geq 1$; see~\cite[Exercise~3.1]{Bugeaud}. For a rational function over
$\Q$ of degree $d$, a classical restriction confines the type of the image
of a $U_m$-number to $\{\lceil m/d\rceil,\ldots,m\}$; see
\cite[Exercise~3.6]{Bugeaud}.

We ask which types actually occur, allowing the coefficients of $R$ to be
real algebraic numbers. Write $R=P/Q$ with $P,Q$ coprime and $Q$ monic,
and let $K_R$ be the field generated over $\Q$ by their coefficients.
This is the \emph{minimal field of definition} of $R$. Put
$d=\deg R$ and $e=[K_R:\Q]$. The classical degree argument now gives
$\lceil m/(ed)\rceil\leq\type(R(\xi))\leq em$ for $\xi\in\U_m$;
we record it as Proposition~\ref{prop:type-bounds}. 

Our main theorem
shows that these bounds are sharp locally: every permitted integer
occurs in every nonempty open interval.

\begin{theorem}\label{thm:local-spectrum}
Let $R\in\overline{\Q}(x)\cap\R(x)$ be nonconstant, put $d=\deg R$ and
$e=[K_R:\Q]$, and let $m\geq 1$ be an integer. For every nonempty open interval
$J\subseteq\R$, we have
\begin{equation}\label{eq:local-spectrum}
  \bigl\{\type(R(\xi)):\xi\in J\cap\U_m\bigr\}
  =\bigl\{r\in\Z:\lceil m/(ed)\rceil\leq r\leq em\bigr\}.
\end{equation}
\end{theorem}

For $R\in\Q(x)$, the spectrum is $\{\lceil m/d\rceil,\ldots,m\}$.
For $m=1$, it is $\{1,\ldots,e\}$, independently of $d$; in particular,
the largest image type determines the degree of the minimal field of
definition. The theorem also settles the preservation problem.

\begin{corollary}\label{cor:classification}
Let $m\geq 1$ be an integer and let
$R\in\overline{\Q}(x)\cap\R(x)$ be nonconstant.
Then $R(\U_m)\subseteq\U_m$ if and only if $R\in\Q(x)$ and either
$m=1$ or $\deg R=1$. In particular, for $m\geq 2$, the preservers are
precisely the M\"obius transformations over $\Q$, and then
$R(\U_m)=\U_m$.
\end{corollary}

\begin{remark}
The field of definition matters even in degree one. For example,
$R(x)=x+\sqrt2$ attains each of the types $1,2,3,4$ on $J\cap\U_2$
for every nonempty open interval $J$. More generally, for every pair of
types permitted by Theorem~\ref{thm:local-spectrum}, the corresponding
points $(\xi,R(\xi))$ are dense on the real graph of $R$ away from its
poles.
\end{remark}

Earlier, Alnia{\c c}\i k \cite{Alniacik} proved that, for every $m\geq 1$,
rationally parametrized curves over $\mathbb Q$ contain infinitely many
points whose coordinates are all $U_m$-numbers. Related work has considered rational or
algebraic functions evaluated at special subclasses of Liouville numbers
defined by stronger approximation properties; see \cite{BM,CMT}.
Here, rather than seeking points at which the $U$-type is preserved, we fix
the rational map and determine locally the complete spectrum of image types
on an arbitrary $U$-class.

For the proof, fix an admissible type $r$. We construct a rapidly
convergent sequence of algebraic numbers for which the pairs of degrees
$(\deg\alpha,\deg R(\alpha))$ alternate between $(m,em)$ and $(edr,r)$.
Rational translates provide the first pair; a norm polynomial and local
Hilbert specialization provide the second. An auxiliary choice of number
field separates the conjugates and forces the full degrees. The gap
principle of Bilu--Marques--Moreira~\cite[Theorem~3.1]{BMM} excludes
lower-degree approximants to the limit. Its conclusion, transferred
through a local inverse of $R$, gives the exact type of the image.

\section{Preliminaries}

\emph{Notation.} We write $A\ll_{\mathcal P}B$ or
$A=O_{\mathcal P}(B)$ if $|A|\leq CB$, where $B\geq 0$ and $C>0$
depends only on the indicated parameters $\mathcal P$. For nonnegative
$A,B$, the notation $A\asymp_{\mathcal P}B$ means both
$A\ll_{\mathcal P}B$ and $B\ll_{\mathcal P}A$. Dependence on fixed data
may be omitted. All algebraic degrees are taken over $\Q$.

\subsection{Heights and rational functions}

We use the following form of Liouville's inequality; see, for instance,
\cite[Appendix~A]{Bugeaud}. If $\alpha\ne\beta$ are algebraic numbers of
degrees $a,b\leq D$, then
\begin{equation}\label{eq:separation}
  |\alpha-\beta|\gg_D H(\alpha)^{-b}H(\beta)^{-a}.
\end{equation}
We also use Northcott's theorem: algebraic numbers of bounded degree and
bounded height form a finite set; see~\cite[Chap.~1]{BombieriGubler}.

Let $\alpha$ be an algebraic number of degree $a$, with primitive minimal
polynomial $a_0\prod_{i=1}^a(X-\alpha_i)\in\Z[X]$. Its \emph{absolute
logarithmic Weil height} is
\[
  \mathsf h(\alpha)=\frac1a\left(
  \log|a_0|+\sum_{i=1}^a\log\max\{1,|\alpha_i|\}\right).
\]
The standard height estimates~\cite[Chap.~1]{BombieriGubler} give
\begin{equation}\label{eq:weil-heights}
  \log H(\alpha)=a\,\mathsf h(\alpha)+O_a(1),
  \qquad
  \mathsf h(R(\alpha))=d\,\mathsf h(\alpha)+O_R(1),
\end{equation}
where $R\in\overline{\Q}(x)$ is nonconstant of degree $d$ and
$R(\alpha)$ is finite. We will also use
\begin{equation}\label{eq:sum-height}
  \mathsf h(\alpha+\beta)
  \leq\mathsf h(\alpha)+\mathsf h(\beta)+\log2.
\end{equation}
For $t=a/q\in\Q$ in lowest terms, with $q>0$, we have
$\mathsf h(t)=\log H(t)$, where $H(t)=\max\{|a|,q\}$.

By~\eqref{eq:weil-heights}, $w_n^*(\xi)=\infty$ if and only if, for
every $A>0$, the inequality
\[
  0<|\xi-\alpha|\leq\exp\bigl(-A\mathsf h(\alpha)\bigr)
\]
holds for infinitely many real algebraic numbers $\alpha$ of degree $n$.
We will use this formulation throughout the proof. Also, if
$a=\deg\alpha\leq D$ and $b=\deg R(\alpha)$, then
\begin{equation}\label{eq:height-map}
  H(R(\alpha))\asymp_{R,D}H(\alpha)^{db/a}.
\end{equation}

Write $R=P/Q$ in reduced form, and put $K=K_R$ and $e=[K:\Q]$.
Since $\alpha$ is a root of $P(X)-R(\alpha)Q(X)$, we have
$[K(\alpha):K(R(\alpha))]\leq d$. Consequently,
\begin{equation}\label{eq:extension-degree}
  \deg R(\alpha)\leq e\deg\alpha,
  \qquad
  \deg\alpha\leq ed\deg R(\alpha).
\end{equation}

For later use, we record the resulting restriction on types. For rational
coefficients, it is the classical consequence of algebraic dependence
recalled in the introduction; see~\cite[Exercise~3.6]{Bugeaud}. We include
the short argument for real algebraic coefficients.

\begin{proposition}\label{prop:type-bounds}
Let $R\in\overline{\Q}(x)\cap\R(x)$ be nonconstant, with $d=\deg R$ and
$e=[K_R:\Q]$. If $\xi\in\U_m$, then $R(\xi)\in\U_j$ for some integer
$j$ satisfying
\[
  \left\lceil\frac{m}{ed}\right\rceil\leq j\leq em.
\]
\end{proposition}

\begin{proof}
Put $\eta=R(\xi)$. Since the poles and critical points of $R$ are algebraic and $\xi$ is
transcendental, $\xi$ is neither a pole nor a critical point of $R$.
Hence $R$ has a Lipschitz local inverse $S$ near $\eta$. Moreover,
$\eta$ is transcendental, for otherwise $P(X)-\eta Q(X)$ would be a
nonzero polynomial with algebraic coefficients vanishing at $\xi$.

Choose real algebraic numbers $\alpha_k$ of degree $m$ such that
$\mathsf h(\alpha_k)\to\infty$ and
\[
0<|\xi-\alpha_k|\leq \exp\bigl(-k\mathsf h(\alpha_k)\bigr).
\]
For large $k$, the points $\alpha_k$ lie in a fixed neighborhood on which
$R$ is Lipschitz. By \eqref{eq:extension-degree} and
\eqref{eq:weil-heights},
\[
\deg R(\alpha_k)\leq em,
\qquad
\mathsf h(R(\alpha_k))
=d\mathsf h(\alpha_k)+O_R(1)\to\infty.
\]
Hence $R(\alpha_k)$ approximates $\eta$ with exponents tending to
infinity. Passing to a subsequence of constant degree $b\leq em$, we
obtain $w_b^*(\eta)=\infty$. Let
\[
j=\min\{\ell\geq1:w_\ell^*(\eta)=\infty\}.
\]
Then $\eta\in\U_j$ and $j\leq b\leq em$.

Conversely, take arbitrarily good algebraic approximants $\beta_k$ of
degree exactly $j$ to $\eta$ and put
$\widetilde{\alpha}_k=S(\beta_k)$ for all sufficiently large $k$. By
\eqref{eq:extension-degree},
\[
\deg\widetilde{\alpha}_k\leq edj,
\]
while the Lipschitz property of $S$ and
\[
\mathsf h(\beta_k)
=d\mathsf h(\widetilde{\alpha}_k)+O_R(1)
\]
show that $\widetilde{\alpha}_k$ approximates $\xi$ with exponents tending
to infinity. Passing again to a subsequence of constant degree, we obtain
$w_a^*(\xi)=\infty$ for some $a\leq edj$. Since $\xi\in\U_m$, the
minimality of $m$ gives $m\leq a$, and therefore
\[
\left\lceil\frac{m}{ed}\right\rceil\leq j\leq em.
\]
\end{proof}

\subsection{Local Hilbert specialization}

We use the following local form of Hilbert's irreducibility theorem.
It allows an irreducible specialization to be chosen arbitrarily close
to a prescribed real number, with quantitative control in terms of height.

\begin{lemma}[Local Hilbert specialization]\label{lem:local-hilbert}
Let $F(T,X)\in\Q[T,X]$ be irreducible in $\Q(T)[X]$, with
$n=\deg_X F\geq 1$. Let $c\in\R$ and $0<\lambda<1$. Then there exist
rational numbers $t$ of arbitrarily large height such that
\[
  0<|t-c|\leq H(t)^{-\lambda}
\]
and $F(t,X)$ is irreducible over $\Q$ of degree $n$.
\end{lemma}

\begin{proof}
By~\cite[Proposition~3.3.5]{Serre}, the exceptional specializations are
contained in a thin set $\Omega\subset\mathbf P^1(\Q)$, including the
zeros of the leading coefficient of $F$ in $X$. Serre's estimate
\cite[Proposition~3.4.2]{Serre} gives
\begin{equation}\label{eq:thin-count}
  \#\{t\in\Omega\cap\Q:H(t)\leq B\}\ll_\Omega B.
\end{equation}

Choose $\lambda<\kappa<1$. For an integer $Q\geq 2$, let
$\mathcal T_Q$ consist of the reduced fractions $a/q$ such that
\[
  Q\leq q<2Q,
  \qquad c<\frac aq<c+Q^{-\kappa}.
\]
For each $q$, the elementary coprimality count gives
\[
  \#\{a\in\Z:cq<a<(c+Q^{-\kappa})q,\ \gcd(a,q)=1\}
  =Q^{-\kappa}\varphi(q)+O(\tau(q)),
\]
where $\varphi$ is Euler's totient function and $\tau$ counts the positive
divisors. We use the classical estimates
\[
  \sum_{q\leq x}\varphi(q)
  =\frac{3}{\pi^2}x^2+O(x\log x),
  \qquad
  \sum_{q\leq x}\tau(q)\ll x\log x;
\]
see~\cite[Chap.~3]{Apostol}. Summing the coprimality count over
$Q\leq q<2Q$, we obtain
\begin{equation}\label{eq:many-rationals}
  \#\mathcal T_Q
  =\frac9{\pi^2}Q^{2-\kappa}+O(Q\log Q)
  \gg Q^{2-\kappa}.
\end{equation}

Every $t\in\mathcal T_Q$ satisfies
\[
  Q\leq H(t)\leq CQ,
  \qquad C=2(1+|c|).
\]
Thus~\eqref{eq:thin-count} excludes at most $O_{\Omega,c}(Q)$ elements
of $\mathcal T_Q$. Since $2-\kappa>1$, the count in
\eqref{eq:many-rationals} is larger for all sufficiently large $Q$.
We may therefore choose $t_Q\in\mathcal T_Q\setminus\Omega$.

By construction, $F(t_Q,X)$ is irreducible over $\Q$ of degree $n$.
Moreover,
\[
  0<(t_Q-c)H(t_Q)^\lambda
  \leq C^\lambda Q^{\lambda-\kappa}<1
\]
for all sufficiently large $Q$, since $\lambda<\kappa$. This gives the
required approximation, while $H(t_Q)\geq Q\to\infty$.
\end{proof}

\subsection{A gap principle}

The exclusion of approximation in smaller degrees will follow from the
gap principle of Bilu--Marques--Moreira~\cite[Theorem~3.1]{BMM}. Related
arguments go back to LeVeque~\cite{LeVeque}; see also
\cite[Theorem~7.4]{Bugeaud}. We state its real case below.

\begin{theorem}[Bilu--Marques--Moreira]\label{thm:gap}
Let $N\geq 1$, and let $z$ be a transcendental real number. Let
$(a_k)_{k\geq 1}$ be a sequence of real algebraic numbers of degree at
most $N$ converging to $z$, with $\mathsf h(a_k)\to\infty$. Suppose that
there are positive real numbers $u_k\to\infty$ and constants $C,A>0$
such that
\begin{equation}\label{eq:gap-hypotheses}
  |z-a_k|\leq C\exp\bigl(-u_k\mathsf h(a_k)\bigr),
  \qquad
  \mathsf h(a_{k+1})\leq A u_k\mathsf h(a_k)
\end{equation}
for every $k\geq 1$. Then there are constants $c,E>0$ such that
\begin{equation}\label{eq:gap-conclusion}
  |z-\beta|\geq c\exp\bigl(-E\mathsf h(\beta)\bigr)
\end{equation}
for every real algebraic number $\beta$ of degree at most $N$ which is
distinct from every term of $(a_k)$.
\end{theorem}

For bounded degrees,~\eqref{eq:weil-heights} turns
\eqref{eq:gap-conclusion} into a uniform lower bound of the form
$c'H(\beta)^{-E'}$. The height divergence follows from Northcott's theorem
and the transcendence of the limit; we retain it explicitly among the
hypotheses. No monotonicity of the heights is required.

\section{Algebraic approximation under rational maps}

Throughout this section, let $R=P/Q\in\overline{\Q}(x)\cap\R(x)$ be
nonconstant, with $P,Q$ coprime and $Q$ monic. Put $K=K_R$,
$e=[K:\Q]$, and $d=\deg R$, and let $L$ be the normal closure of $K/\Q$.
For an embedding $\sigma:K\hookrightarrow L$, write $R^\sigma$ for the
rational function obtained by applying $\sigma$ to its coefficients.
The normalization makes $P,Q$ unique. Thus $R^\sigma=R^\tau$ forces
equality on every coefficient of $P,Q$, and hence $\sigma=\tau$.

If $I$ is a nonempty bounded open interval whose closure contains no pole
or critical point of $R$, there are constants $0<c_0\leq M$ such that
\begin{equation}\label{eq:lipschitz}
  c_0|x-y|\leq|R(x)-R(y)|\leq M|x-y|
  \qquad(x,y\in I).
\end{equation}
Indeed, $R'$ is continuous and nonzero on $\overline I$. Thus $R|_I$ is
one-to-one, $R(I)$ is open, and its inverse $S:R(I)\to I$ is
$c_0^{-1}$-Lipschitz.

\subsection{Separating the conjugates}

To obtain the full degrees in the two constructions below, we must rule
out coincidences among conjugates of certain translated rational
functions. We do this by choosing the auxiliary algebraic number to have
the same degree over $L$ as over $\Q$.

\begin{lemma}\label{lem:conjugates}
For every integer $n\geq 1$, there is a real algebraic number $\theta$ of
degree $n$ whose minimal polynomial remains irreducible over $L$. For
any such $\theta$, with conjugates $\theta_1,\ldots,\theta_n$, the
functions in each of the two families
\begin{equation}\label{eq:conjugate-families}
  \bigl\{R^\sigma(X+\theta_i)\bigr\}_{\sigma,i},
  \qquad
  \bigl\{R^\sigma(X)-\theta_i\bigr\}_{\sigma,i}
\end{equation}
are pairwise distinct, where $\sigma$ ranges over the embeddings
$K\hookrightarrow L$ and $1\leq i\leq n$.
\end{lemma}

\begin{proof}
Choose a rational prime $p$ unramified in $L$ and take the positive real
root $\theta=p^{1/n}$. At every prime of $L$ above $p$, the valuation of
$p$ is one. Thus $X^n-p$ is Eisenstein at such a prime and is irreducible
over $L$.

Now let $\theta$ be any number as in the statement, and let
$g\in\Q[X]$ be its monic minimal polynomial. We first show that
\begin{equation}\label{eq:conjugate-difference}
  \theta_j-\theta_i\in L\quad\Longrightarrow\quad i=j.
\end{equation}
Indeed, if $a=\theta_j-\theta_i\in L$, then $g(X)$ and $g(X+a)$ have
the common root $\theta_i$. Since $g$ is irreducible over $L$ and both
polynomials are monic of degree $n$, we obtain $g(X+a)=g(X)$.
Comparing the coefficients of $X^{n-1}$ gives $na=0$, so $a=0$.
The conjugates are distinct, and~\eqref{eq:conjugate-difference} follows.

We also need a simple observation about translations. A nonconstant
rational function has no nonzero additive period: a finite pole would
generate infinitely many poles, while a nonconstant polynomial cannot be
periodic. Suppose that $f,g\in L(X)$ are nonconstant and
$f(X+a)=g(X)$ for some $a\in\overline{\Q}$. For every
$\rho\in\operatorname{Gal}(\overline{\Q}/L)$, we have
\[
  f(X+\rho(a))=g(X)=f(X+a).
\]
Hence $\rho(a)-a$ is an additive period of $f$, and therefore
$\rho(a)=a$. It follows that $a\in L$.

Suppose now that two members of the first family in
\eqref{eq:conjugate-families} coincide. After translating the variable,
this gives
\[
  R^\sigma(X)=R^\tau(X+\theta_j-\theta_i).
\]
The preceding observation implies $\theta_j-\theta_i\in L$.
By~\eqref{eq:conjugate-difference}, $i=j$, and then the minimality of
$K$ gives $\sigma=\tau$.

Finally, an equality in the second family gives
\[
  R^\sigma(X)-R^\tau(X)=\theta_i-\theta_j.
\]
The left-hand side belongs to $L(X)$ and is constant, so this constant
lies in $L$. Once again,~\eqref{eq:conjugate-difference} gives $i=j$,
and then $\sigma=\tau$.
\end{proof}

\subsection{Rational translates}

We first construct nearby points of degree $m$ whose images have the
full degree $em$. For $m\geq 1$, put
\[
  A_m(I)=\{\alpha\in I:\deg\alpha=m,\ \deg R(\alpha)=em\}.
\]

\begin{lemma}\label{lem:translates}
Let $I$ be as above, let $m\geq 1$, and let $b\in I$. For every
$0<\sigma<1$, there are elements $\alpha\in A_m(I)$ of arbitrarily
large height such that
\begin{equation}\label{eq:translate-height-distance}
  0<|b-\alpha|\leq\exp\bigl(-\sigma\mathsf h(\alpha)\bigr).
\end{equation}
\end{lemma}

\begin{proof}
Choose $\theta$ of degree $m$ as in Lemma~\ref{lem:conjugates}. Its
minimal polynomial remains irreducible over $L$, and hence over $K$.
Therefore
\[
  [K(\theta):\Q]
  =[K(\theta):K]\,[K:\Q]=em.
\]
Every embedding $\tau:K\hookrightarrow L$ extends to $m$ embeddings of
$K(\theta)$, sending $\theta$ to each of its conjugates
$\theta_1,\ldots,\theta_m$. For $t\in\Q$, the conjugate values of
$R(\theta+t)$ are consequently
\[
  R^\tau(\theta_i+t),
  \qquad \tau:K\hookrightarrow L,\quad 1\leq i\leq m.
\]

By Lemma~\ref{lem:conjugates}, the corresponding rational functions of
$t$ are pairwise distinct. Each equality between two of their values
therefore occurs at only finitely many specializations. Excluding these
specializations and the poles, we obtain a finite exceptional set
$\mathcal E\subset\overline{\Q}$ outside which all $em$ values are
finite and distinct. Since $R(\theta+t)\in K(\theta)$, this proves
\begin{equation}\label{eq:translate-degrees}
  \deg(\theta+t)=m,
  \qquad \deg R(\theta+t)=em
  \qquad(t\in\Q\setminus\mathcal E).
\end{equation}
Here the first equality follows from $\Q(\theta+t)=\Q(\theta)$.

For every integer $q\geq 1$, set
\[
  t_q=\frac{\lfloor q(b-\theta)\rfloor+1}{q},
  \qquad \alpha_q=\theta+t_q.
\]
Then
\begin{equation}\label{eq:translate-distance}
  0<\alpha_q-b\leq q^{-1},
  \qquad H(t_q)\ll_{b,\theta}q.
\end{equation}
The first inequality is strict even when $q(b-\theta)$ is an integer.
Thus $t_q\to b-\theta$ without equality, so the finite set $\mathcal E$
is avoided for all sufficiently large $q$. Also, $\alpha_q\in I$
eventually. Hence $\alpha_q\in A_m(I)$ for all sufficiently large $q$.

By Northcott's theorem, $\mathsf h(\alpha_q)\to\infty$. Otherwise a
bounded-height subsequence would take values in a finite set, which is
impossible since $\alpha_q\to b$ and $\alpha_q\ne b$. On the other hand,
\eqref{eq:sum-height} and~\eqref{eq:translate-distance} give
\[
  \mathsf h(\alpha_q)
  \leq\mathsf h(\theta)+\log H(t_q)+\log2
  \leq\log q+O_{b,\theta}(1).
\]
For a suitable constant $C>0$, we therefore have
\[
  0<|\alpha_q-b|\exp\bigl(\sigma\mathsf h(\alpha_q)\bigr)
  \leq Cq^{\sigma-1}<1
\]
for all sufficiently large $q$, since $\sigma<1$. This proves
\eqref{eq:translate-height-distance}.
\end{proof}

\subsection{Full-degree preimages}

We next construct nearby points whose images have degree $r$, while the
points themselves have the full degree $edr$.

Fix an integer $r\geq 1$, choose $\vartheta$ of degree $r$ as in
Lemma~\ref{lem:conjugates}, and put $E=K(\vartheta)$. Then
$[E:\Q]=er$. Consider the norm polynomial
\begin{equation}\label{eq:generic-polynomial}
  \Phi(T,X)
  =\operatorname{N}_{E/\Q}\bigl(P(X)-(T+\vartheta)Q(X)\bigr)
  \in\Q[T,X].
\end{equation}
Here the norm is the product over all $\Q$-embeddings of $E$, acting on
the coefficients and fixing $T,X$. Writing
$\vartheta_1,\ldots,\vartheta_r$ for the conjugates of $\vartheta$, we
may express it as
\[
  \Phi(T,X)
  =\prod_{\sigma:K\hookrightarrow L}\prod_{i=1}^r
    \bigl(P^\sigma(X)-(T+\vartheta_i)Q^\sigma(X)\bigr).
\]
Each factor is nonzero and has degree at most $d$ in $X$. Thus $\Phi$
is nonzero and $\deg_X\Phi\leq edr$. The next lemma supplies the irreducibility needed for local Hilbert
specialization to preserve the full degree $edr$.

\begin{lemma}\label{lem:generic-irreducible}
The polynomial $\Phi(T,X)$ is irreducible in $\Q(T)[X]$ and has degree
$edr$ in $X$.
\end{lemma}

\begin{proof}
Let $x$ be transcendental over $E$, and set $\tau=R(x)-\vartheta$.
Since $R$ is nonconstant, $\tau$ is transcendental over $E$. The $er$
embeddings of $E(x)$ over $\Q(x)$ send $\tau$ to the values
\[
  R^\sigma(x)-\vartheta_i,
  \qquad \sigma:K\hookrightarrow L,\quad 1\leq i\leq r.
\]
These are pairwise distinct by Lemma~\ref{lem:conjugates}. Consequently,
\[
  [\Q(x,\tau):\Q(x)]=er=[E(x):\Q(x)].
\]
Together with the inclusion $\Q(x,\tau)\subseteq E(x)$, this proves
\begin{equation}\label{eq:primitive-function-field}
  \Q(x,\tau)=E(x).
\end{equation}

We now compute the degree over $\Q(\tau)$. Since
$E(\tau)=E(R(x))$, the degree formula for rational maps gives
$[E(x):E(\tau)]=d$. Also, adjoining the transcendental element $\tau$
preserves the degree of the constant field extension, so
$[E(\tau):\Q(\tau)]=[E:\Q]=er$. Using
\eqref{eq:primitive-function-field} and the tower law, we obtain
\begin{equation}\label{eq:generic-degree}
  [\Q(\tau,x):\Q(\tau)]
  =[E(x):E(\tau)]\,[E(\tau):\Q(\tau)]
  =edr.
\end{equation}

By the definition of $\tau$, the factor
$P(x)-(\tau+\vartheta)Q(x)$ vanishes, and hence $\Phi(\tau,x)=0$.
The polynomial $\Phi(\tau,X)$ is nonzero, because $\tau$ is
transcendental over $\Q$, and its degree in $X$ is at most $edr$.
By~\eqref{eq:generic-degree}, the minimal polynomial of $x$ over
$\Q(\tau)$ has degree $edr$ and divides $\Phi(\tau,X)$. The two
polynomials are therefore scalar multiples of one another. Identifying
$\Q(T)$ with $\Q(\tau)$ proves both assertions.
\end{proof}

Combining the preceding lemma with local Hilbert specialization and a
local inverse of $R$, we obtain preimages of full degree with an
approximation bound in terms of height.

\begin{proposition}\label{prop:full-degree}
Let $x_0\in\R$ be neither a pole nor a critical point of $R$, and let
$r\geq 1$ be an integer. For every $0<\sigma<d$, there are real algebraic numbers
$\beta$, not poles of $R$ and of arbitrarily large height, such that
\begin{equation}\label{eq:full-degree-assertion}
\begin{gathered}
  \deg\beta=edr,\qquad \deg R(\beta)=r,\\
  0<|\beta-x_0|\leq\exp\bigl(-\sigma\mathsf h(\beta)\bigr).
\end{gathered}
\end{equation}
For these numbers,
\begin{equation}\label{eq:full-degree-height}
  H(R(\beta))\asymp_{R,r}H(\beta)^{1/e}.
\end{equation}
\end{proposition}

\begin{proof}
Choose a bounded open interval $I$ containing $x_0$ whose closure contains
no pole or critical point of $R$, and let $S:R(I)\to I$ be the inverse
from~\eqref{eq:lipschitz}. Fix $0<\sigma<d$. With $\vartheta$ and
$\Phi$ as above, put $c=R(x_0)-\vartheta$ and choose
$\sigma/d<\lambda<1$. By Lemmas~\ref{lem:generic-irreducible} and~\ref{lem:local-hilbert},
there are rational numbers $t$ of arbitrarily large height such that
\begin{equation}\label{eq:t-approximation}
  0<|t-c|\leq H(t)^{-\lambda}
\end{equation}
and $\Phi(t,X)$ is irreducible over $\Q$ of degree $edr$.

Set $\delta=\vartheta+t$. Since $t\in\Q$, we have
$\Q(\delta)=\Q(\vartheta)$ and hence $\deg\delta=r$. Moreover,
\[
  0<|\delta-R(x_0)|=|t-c|\leq H(t)^{-\lambda}.
\]
Thus $\delta\to R(x_0)$, and $\delta\in R(I)$ for all sufficiently
large heights. Define $\beta=S(\delta)\in I$. Then $R(\beta)=\delta$
and $\beta\ne x_0$. The nonzero polynomial
$P(X)-\delta Q(X)$ vanishes at $\beta$, so $\beta$ is algebraic.
Furthermore,
\[
  P(\beta)-(t+\vartheta)Q(\beta)=0.
\]
This is one factor in the norm defining $\Phi(t,\beta)$, and therefore
$\Phi(t,\beta)=0$. Since $\Phi(t,X)$ is irreducible over $\Q$ of degree
$edr$, we obtain
\begin{equation}\label{eq:preimage-degrees}
  \deg\beta=edr,
  \qquad \deg R(\beta)=r.
\end{equation}
The exponent in~\eqref{eq:height-map} is $dr/(edr)=1/e$, which proves
\eqref{eq:full-degree-height}.

It remains to compare the approximation error with the height of $\beta$.
Apply~\eqref{eq:sum-height} to both $\delta=\vartheta+t$ and
$t=\delta-\vartheta$. Since $\mathsf h(t)=\log H(t)$, this bounds
$\mathsf h(\delta)-\log H(t)$ independently of $t$. Together with
\eqref{eq:weil-heights}, we obtain
\begin{equation}\label{eq:preimage-parameter-height}
\begin{aligned}
  \mathsf h(\delta)&=\log H(t)+O_\vartheta(1),\\
  \mathsf h(\beta)&=\frac1d\log H(t)+O_{R,\vartheta}(1).
\end{aligned}
\end{equation}
In particular, $\mathsf h(\beta)\to\infty$. The Lipschitz property of
$S$ and~\eqref{eq:t-approximation} give
\[
  0<|\beta-x_0|
  \leq c_0^{-1}|\delta-R(x_0)|
  \leq c_0^{-1}H(t)^{-\lambda}.
\]
Hence there is a constant $C>0$, independent of $t$, such that
\[
  0<|\beta-x_0|\exp\bigl(\sigma\mathsf h(\beta)\bigr)
  \leq C H(t)^{\sigma/d-\lambda}<1
\]
for all sufficiently large heights, since $\sigma/d<\lambda$.
This proves~\eqref{eq:full-degree-assertion}. The center $x_0$ is fixed
before the specializations are chosen; no height threshold uniform in
$x_0$ is required.
\end{proof}

\section{Proof of Theorem~\ref{thm:local-spectrum}}

Proposition~\ref{prop:type-bounds} gives one inclusion in
\eqref{eq:local-spectrum}. For the reverse inclusion, fix $R,m,J$ as in
the theorem, write $K=K_R$, $e=[K:\Q]$ and $d=\deg R$, and choose an
integer $r$ with
\begin{equation}\label{eq:admissible-r}
  \left\lceil\frac{m}{ed}\right\rceil\leq r\leq em.
\end{equation}
Since $R$ has only finitely many poles and critical points, we may choose
a nonempty bounded open interval $I$ with $\overline I\subset J$ such
that $R$ has neither on $\overline I$. Let $c_0,M,S$ be as in
\eqref{eq:lipschitz}, put $D=edr$, and define
\begin{equation}\label{eq:sets}
\begin{aligned}
  A&=A_m(I)=\{\alpha\in I:\deg\alpha=m,\ \deg R(\alpha)=em\},\\
  B_r&=\{\beta\in I:\deg\beta=D,\ \deg R(\beta)=r\}.
\end{aligned}
\end{equation}
The admissibility conditions are exactly $D\geq m$ and $em\geq r$.
They ensure that the degrees in the original sequence are never below
$m$, and those in its image are never below $r$.

Taking $\sigma=1/2$ in Lemma~\ref{lem:translates} and
Proposition~\ref{prop:full-degree}, we can approximate every fixed point
$b\in I$ by points $z$ in either $A$ or $B_r$, with arbitrarily large
$\mathsf h(z)$ and
\[
  0<|z-b|\leq\exp\bigl(-\mathsf h(z)/2\bigr).
\]
We now alternate between these two families.

\subsection{The alternating sequence}

Choose $\gamma_1\in A$ with $h_1:=\mathsf h(\gamma_1)\geq 1$.
Having chosen $\gamma_k$, choose $\gamma_{k+1}\in B_r$ when $k$ is odd
and $\gamma_{k+1}\in A$ when $k$ is even. Writing
$h_{k+1}=\mathsf h(\gamma_{k+1})$, we may require
\begin{equation}\label{eq:recursive-choices}
\begin{aligned}
  h_{k+1}&\geq 2k h_k+2\log2,\\
  0<|\gamma_{k+1}-\gamma_k|
  &\leq\exp(-h_{k+1}/2).
\end{aligned}
\end{equation}
At the first step, require also
\[
  2\exp(-h_2/2)<\dist(\gamma_1,\partial I).
\]
Each center is fixed before the next term is chosen, and both local
approximation results allow its height to be arbitrarily large. Thus all
these requirements can be imposed simultaneously.

The first inequality in~\eqref{eq:recursive-choices} ensures that
$h_k\to\infty$ and gives
\[
  \exp(-h_{k+1}/2)\leq\frac12\exp(-h_k/2).
\]
Consequently,
\[
  \sum_{j=k}^{\infty}|\gamma_{j+1}-\gamma_j|
  \leq\sum_{j=k}^{\infty}\exp(-h_{j+1}/2)
  \leq 2\exp(-h_{k+1}/2).
\]
It follows that $(\gamma_k)$ converges to a real number $\xi$, with
\begin{equation}\label{eq:limit-tail}
  |\xi-\gamma_k|
  \leq 2\exp(-h_{k+1}/2)
  \leq\exp(-k h_k)
  \qquad(k\geq 1).
\end{equation}
The last inequality follows again from~\eqref{eq:recursive-choices}.
The extra condition at the first step gives
$|\xi-\gamma_1|<\dist(\gamma_1,\partial I)$, so $\xi\in I\subset J$.

The number $\xi$ is transcendental. Otherwise, put $s=\deg\xi$.
Since $h_k\to\infty$, we have $\gamma_k\ne\xi$ for all sufficiently
large $k$. The degrees of the $\gamma_k$ are at most $D$, so
\eqref{eq:separation} and~\eqref{eq:weil-heights} give
\[
  |\xi-\gamma_k|
  \gg_{\xi,D}H(\gamma_k)^{-s}
  \gg_{\xi,D}\exp(-sD h_k).
\]
For $k>sD$, this contradicts~\eqref{eq:limit-tail} as $k\to\infty$.

\subsection{The type of the limit}

At odd indices, $\deg\gamma_k=m$. Since $\xi$ is transcendental and
$h_k\to\infty$, the estimates in~\eqref{eq:limit-tail} give
approximation exponents tending to infinity along these indices.
By~\eqref{eq:weil-heights},
\begin{equation}\label{eq:source-infinite}
  w_m^*(\xi)=\infty.
\end{equation}
We now use the gap principle to exclude all smaller degrees, retaining
its full conclusion for the analysis of the image.

Set
\[
  u_k=\frac{h_{k+1}}{2h_k}.
\]
By~\eqref{eq:recursive-choices}, $u_k\geq k$, and hence $u_k\to\infty$.
Moreover,
\[
  |\xi-\gamma_k|\leq 2\exp(-u_kh_k),
  \qquad h_{k+1}=2u_kh_k.
\]
We have already proved that $\xi$ is transcendental and $h_k\to\infty$.
Thus every hypothesis of Theorem~\ref{thm:gap} holds with degree bound
$N=D$ and constants $C=A=2$. There are therefore constants $c,E_0>0$
such that
\begin{equation}\label{eq:source-measure}
  |\xi-\alpha|\geq c\exp\bigl(-E_0\mathsf h(\alpha)\bigr)
\end{equation}
for every real algebraic number $\alpha$ of degree at most $D$ which is
distinct from every $\gamma_k$.

All terms $\gamma_k$ have degree at least $m$: the odd terms have degree
$m$, and the even terms have degree $D\geq m$. Hence every real algebraic
number of degree less than $m$ is distinct from all the $\gamma_k$.
Applying~\eqref{eq:source-measure} and then~\eqref{eq:weil-heights}, we
obtain
\[
  w_j^*(\xi)<\infty
  \qquad(1\leq j<m).
\]
Together with~\eqref{eq:source-infinite}, this proves
\begin{equation}\label{eq:source-type}
  \xi\in\U_m.
\end{equation}
The argument includes $m=1$, when there are no smaller degrees to exclude.

\subsection{The type of the image}

Put $\eta=R(\xi)$ and $\delta_k=R(\gamma_k)$. The number $\eta$ is
transcendental: otherwise the nonzero polynomial $P(X)-\eta Q(X)$,
with algebraic coefficients, would vanish at $\xi$. By construction,
\begin{equation}\label{eq:image-degrees}
  \deg\delta_{2j-1}=em,
  \qquad \deg\delta_{2j}=r.
\end{equation}
Also,~\eqref{eq:weil-heights} gives
\[
  \mathsf h(\delta_k)=d h_k+O_R(1).
\]
In particular, these heights tend to infinity. For each fixed $A>0$,
\eqref{eq:lipschitz}, \eqref{eq:limit-tail}, and~\eqref{eq:weil-heights}
give a constant $C>0$, depending only on $A,R,I$, such that
\[
\begin{aligned}
  0<|\eta-\delta_{2j}|\exp\bigl(A\mathsf h(\delta_{2j})\bigr)
  &\leq M\exp\bigl(-2j h_{2j}+A\mathsf h(\delta_{2j})\bigr)\\
  &\leq C\exp\bigl(-(2j-Ad)h_{2j}\bigr)<1
\end{aligned}
\]
for all sufficiently large $j$. The even terms have degree $r$, so
\begin{equation}\label{eq:image-infinite}
  w_r^*(\eta)=\infty.
\end{equation}
If $r=1$, this already gives $\eta\in\U_1$. Suppose henceforth that
$r\geq 2$.

It remains to exclude approximation of $\eta$ by algebraic numbers of
degree below $r$. We transfer the lower bound~\eqref{eq:source-measure}
through the local inverse $S$. Choose $\varepsilon>0$ such that
$(\eta-\varepsilon,\eta+\varepsilon)\subset R(I)$, and let $\beta$ be a
real algebraic number with $\deg\beta\leq r-1$.

Suppose first that $|\eta-\beta|<\varepsilon$, and set
$\alpha=S(\beta)\in I$. The relation
\[
  P(\alpha)-\beta Q(\alpha)=0
\]
shows that $\alpha$ is algebraic. Since $R(\alpha)=\beta$, we have
$K(\beta)\subseteq K(\alpha)$ and
$[K(\alpha):K(\beta)]\leq d$. Consequently,
\begin{equation}\label{eq:pullback-degree}
\begin{aligned}
  \deg\alpha
  &\leq[K(\alpha):\Q]
  \leq d[K(\beta):\Q]\\
  &\leq ed\deg\beta
  \leq ed(r-1)<D.
\end{aligned}
\end{equation}
We also have $\alpha\ne\gamma_k$ for every $k$. Indeed, equality would
give $\beta=R(\gamma_k)=\delta_k$, whereas
$\deg\delta_k\in\{em,r\}$ and both degrees are at least $r>\deg\beta$.
Thus~\eqref{eq:source-measure} applies to $\alpha$, and
\eqref{eq:lipschitz} yields
\[
\begin{aligned}
  |\eta-\beta|
  &=|R(\xi)-R(\alpha)|\\
  &\geq c_0|\xi-\alpha|
  \geq c_0c\exp\bigl(-E_0\mathsf h(\alpha)\bigr).
\end{aligned}
\]
By~\eqref{eq:weil-heights},
$\mathsf h(\alpha)=d^{-1}\mathsf h(\beta)+O_R(1)$. Thus the preceding
lower bound is at least a fixed positive constant times
$\exp(-(E_0/d)\mathsf h(\beta))$.

If $|\eta-\beta|\geq\varepsilon$, the same estimate holds after decreasing
the constant, since $\mathsf h(\beta)\geq 0$. Consequently, there is
$c'>0$, independent of $\beta$, such that
\begin{equation}\label{eq:image-measure}
  |\eta-\beta|\geq c'\exp\left(-\frac{E_0}{d}\mathsf h(\beta)\right)
  \qquad(\deg\beta\leq r-1).
\end{equation}
By~\eqref{eq:weil-heights}, this implies
\[
  w_j^*(\eta)<\infty
  \qquad(1\leq j<r).
\]
Together with~\eqref{eq:image-infinite}, we conclude that
$R(\xi)=\eta\in\U_r$. This completes the proof.\qed

\begin{proof}[Proof of Corollary~\ref{cor:classification}]
If $e>1$, Theorem~\ref{thm:local-spectrum} realizes the image type
$em>m$, so preservation requires $e=1$, or equivalently $R\in\Q(x)$.
If also $m,d\geq 2$, the same theorem realizes the type $m-1$, and
preservation again fails. Conversely, Proposition~\ref{prop:type-bounds}
gives preservation when $e=1$ and either $m=1$ or $d=1$. In the
M\"obius case, applying it also to $R^{-1}$ gives equality.
\end{proof}

\section*{Acknowledgments}

The author was supported by the Brazilian National Council for Scientific
and Technological Development (CNPq), Grant No.~304467/2023-5.

\end{document}